\documentclass{amsart}[12pt]
\usepackage{amsmath,amssymb,amsthm,hyperref,mathtools}
\usepackage{color}
\usepackage{cleveref}

\usepackage{tikz}
\newtheorem{thm}{Theorem}[section]
\newtheorem{cor}[thm]{Corollary}

\newtheorem{lem}[thm]{Lemma}
\newtheorem{prop}[thm]{Proposition}
\theoremstyle{remark}

\theoremstyle{definition}

\newtheorem{defn}[thm]{Definition}
\newtheorem{notn}[thm]{Notation}

\title{Walecki tournaments with an arc that lies in a unique directed triangle}

\author{Joy Morris}

\thanks{Supported by the Natural Science and Engineering Research Council of Canada (grant RGPIN-2024-04013).}

\address{Department of Mathematics and Computer Science\\
	University of Lethbridge\\
	Lethbridge, AB T1K 3M4\\
	Canada}

\email{joy.morris@uleth.ca}

\begin{document}

\begin{abstract}
A Walecki tournament is any tournament that can be formed by choosing an orientation for each of the Hamilton cycles in the Walecki decomposition of a complete graph on an odd number of vertices. In this paper, we show that if some arc in a Walecki tournament on at least $7$ vertices lies in exactly one directed triangle, then there is a vertex of the tournament (the vertex typically labelled $*$ in the decomposition) that is fixed under every automorphism of the tournament. Furthermore, any isomorphism between such Walecki tournaments maps the vertex labelled $*$ in one to the vertex labelled $*$ in the other.

We also show that among Walecki tournaments with a signature of even length $2k$, of the $2^{2k}$ possible signatures, at least $2^k$ produce tournaments that have an arc that lies in a unique directed triangle (and therefore to which our result applies).
\end{abstract}

\keywords{Tournaments, Walecki decomposition, Walecki tournaments, automorphisms, isomorphisms}

\subjclass[2020]{05C20,05C38,05C45,05C60}

\maketitle

\begin{center}
\textit{To Brian Alspach, who has an enduring soft spot for this problem.}
\end{center}

\section{Introduction}

Walecki tournaments were introduced by Alspach in his PhD thesis in 1966 \cite{alspach:1}. They are orientations of the complete graph $K_n$ that arise from Walecki's elegant decomposition of $K_n$ into Hamilton cycles, when $n$ is odd. 

More precisely, given an integer $m \ge 1$, take $n=2m$ and $N=2m+1$. We identify the vertices of $K_N$ with the elements of $\mathbb Z_n \cup \{*\}$. Define the permutation $$\rho=(0\ 1\ 2 \ \cdots \ n-1)$$ and the Hamilton cycle $$C_0=[*, 0, 1, -1, 2, -2, \ldots, m, *].$$ Take $C_0^+$ to be the cycle $C_0$ oriented in the direction we have given, and $C_0^-$ to be the opposite orientation. Define $C_j=C_0\rho^j$, with $C_j^+$ and $C_j^-$ defined accordingly. Note that $C_j=C_{j+m}$, so in the remainder of the paper when we are considering undirected cycles we may take the subscript modulo $m$, but $C_j^-=C_{j+m}^+$. For any binary string $u \in \mathbb Z_2^m$, with entries $u_0, \ldots, u_{m-1}$, we can form a tournament $W_u$ on the vertices of $K_N$ by choosing the arcs that are in $C_j^+$ if $u_j=1$, and the arcs that are in $C_j^-$ if $u_j=0$. For each string $u \in \mathbb Z_2^m$, this defines a tournament that we call a Walecki tournament. We refer to the binary string $u$ as the \emph{signature} of the tournament.

The initial motivation for studying Walecki tournaments was a conjecture by Kelly that every regular tournament can be decomposed into directed Hamilton cycles. Walecki tournaments are examples of regular tournaments that admit a particularly symmetric decomposition into directed Hamilton cycles, by construction.

Although Walecki tournaments have not been much studied, research on them has focussed on understanding their automorphisms and isomorphisms. One of the first observations about isomorphisms of Walecki tournaments arises from the so-called complementing register shift map acting on the signature of the tournament. For a binary string $u \in \mathbb Z_2^m$, define
$$uR_1=(1-u_m)u_0u_1\cdots u_m-1,\text{ and }R_i=R_1^i.$$
That is, the complementing register shift $R_i$ shifts every entry of the binary string $i$ positions to the right cyclically, and if that results in the entry wrapping past the final entry back to the beginning, then the value of the entry is changed (to its complement). Alspach made the following observation in \cite{alspach:1}

\begin{prop}[Alspach, \cite{alspach:1}]\label{shift}
The tournaments $W_u$ and $W_{uR_i}$ are isomorphic for any integer $i$.
\end{prop}

This was the first observation suggesting that isomorphisms and automorphisms of Walecki tournaments may bear a close relationship to symmetries and anti-symmetries of the signature. This focus continued in \cite{ales:1, ales:2} and related unpublished work by Ales. Ales' work focuses on the situation where there is periodicity in the signature; that is, Walecki tournaments in which the signature can be broken down as the concatenation of some number of copies of either $v$ or $\bar{v}$, where $v$ is a shorter binary string and $\bar{v}$ is the binary string obtained by replacing every $1$ in $v$ with a $0$, and every $0$ with a $1$. Let $1_m$ denote the binary string with $m$ entries all of which are $1$. In~\cite{ales:2}, the automorphisms of $W_{1_m}$ are completely determined (and therefore, by~\Cref{shift}, so are the automorphisms of $W_{1_mR_i}$ for every integer $i$). The map
$$\sigma=(0 \ 1\ 2\ \cdots \ m-1)(n-1\ n-2 \ n-3 \ \cdots\  m)$$
is important in this result and more broadly in automorphisms found in Ales' work.

\begin{thm}[Ales, \cite{ales:2}, Theorem 3.7]\label{0-sig}
The automorphism group of $W_{1_m}$ is as follows:
\begin{enumerate}
\item if $m=1$, then $W_1$ is a directed cycle of length $3$ and its automorphism group is $C_3=\langle (0 \ 1\ *)\rangle$;
\item if $m=2$, then $W_{11}$ lies in the unique isomorphism class of Eulerian orientations of $K_5$ and its automorphism group is $C_5=\langle (0\ 1\ 3\ 2\ *) \rangle$;
\item if $m \ge 3$ is odd, then the automorphism group of $W_{1_m}$ is $C_m =\langle \sigma \rangle$; and
\item if $m \ge 4$ is even, then the automorphism group of $W_{1_m}$ is trivial (that is, $W_{1_m}$ is asymmetric).
\end{enumerate}
\end{thm}

In Ales' work, and in unpublished work the author undertook with Ales, Alspach, and Steve Wilson, it has seemed possible that with a few small exceptions, all automorphisms of Walecki tournaments might arise from the two basic permutations: $\rho$ and $\sigma$;
in fact all of the automorphisms that were found in that previous work were powers of $\rho$ or $\sigma$, or some power of $\sigma$ conjugated by some power of $\rho$ (again, with the two exceptions listed in~\Cref{0-sig}, and two other small exceptions).

This idea supported a conjecture that Alspach had made verbally: that again with a few small exceptions, every automorphism of a Walecki tournament $W_u$ fixes the vertex $*$, and so every isomorphism between Walecki tournaments maps the vertex labelled $*$ in one, to the vertex labelled $*$ in the other. Proving this conjecture (if it is even true) seems to be the ``hard" part of characterising automorphisms and isomorphisms for Walecki tournaments. The other two exceptional cases are as follows:
$W_{101}$ has an affine group of order $21$ acting transitively on its vertices; and the automorphism group of $W_{1011}$ is isomorphic to $C_3$, fixing 3 vertices and acting as two disjoint cycles of length $3$ on the remaining $6$ vertices. (This last action does fix the vertex $*$, and in fact falls into the situation we consider in this paper.) 

In this paper, we present a class of signatures such that the automorphism group of any Walecki tournament with one of these signatures fixes at least three of the vertices of the tournament, including the vertex $*$. It may be the case that with the exception of $W_{1011}$, the tournaments with these signatures all have trivial automorphism groups; certainly any automorphisms they may have do not arise directly from just $\rho$ and $\sigma$. 

In Section 2, we consider in some detail ways of understanding the triangles that contain a particular arc, and present some elementary results about the numbers of triangles that take various forms. In Sections 3, 4, and 5, we consider various locations in a Walecki tournament in which an arc may lie, and show that in almost all of these locations, any arc we look at must lie in more than one directed triangle. In fact, in all cases we conclude that if there is a unique directed triangle containing our arc $a$, then the third vertex of that triangle must be $*$; this central result of the paper is pulled together and proved in Section 6. In Section 7, we determine a family of signatures that produce Walecki tournament that have an arc that lies in a unique directed triangle, and prove that this is the case. We conclude the paper with some broad observations and open problems.

\section{Triangle types}

In a tournament, there are two basic types of cycles of length $3$ that can appear. If the subdigraph induced on the three vertices is regular, we refer to this as a \emph{directed triangle}. If not, then we refer to it as a \emph{transitive triangle}. 

From the perspective of any arc in the triangle, a directed triangle appears the same: the other two arcs form a directed path from the head of the given arc to the tail of the given arc. However, in a transitive triangle, each arc has a unique role. From the perspective of one arc, both of the other arcs point away, to the remaining vertex. We refer to this situation as an ``out" triangle. From the perspective of a different arc, both of the other arcs point toward this arc, from the remaining vertex. We refer to this situation as an ``in" triangle. Finally, from the perspective of the remaining arc, the other two arcs form a directed path from the tail of the given arc to its head; we refer to this situation as a ``bypass" triangle, adopting this terminology from~\cite{ARR}. In that paper, the arcs whose perspective we are considering here are referred to as bypass arcs, on the basis that such an arc forms a direct route from its starting vertex to its terminal vertex, bypassing the other vertex involved in the directed path of length $2$ given by the other two arcs of the bypass triangle. Often we use this more intuitive description that $a$ is a bypass arc in a particular triangle, rather than saying that the triangle is a bypass triangle from the perspective of $a$.

We establish some notation for the numbers of triangles of each of these types that contain a given arc. Since this is a key concept of this paper, this notation will be used in the statements of most of our results.

\begin{notn}
Let $\Gamma$ be a tournament, and let $a$ be an arc of $\Gamma$. Then we will use: 
\begin{itemize}
\item $i(a)$ to  denote the number of triangles containing $a$ that are ``in" triangles from the perspective of $a$;
\item $o(a)$ to denote the number of triangles containing $a$ that are ``out" triangles from the perspective of $a$; 
\item $b(a)$ to denote the number of triangles that are ``bypass" triangles from the perspective of $a$; and 
\item $d(a)$ to denote the number of triangles containing $a$ that are directed triangles from the perspective of $a$.
\end{itemize}
\end{notn}

There are some very nice relationships between these parameters in any regular tournament.

\begin{lem}\label{triangle-types}
Let $\Gamma$ be a regular tournament on $2m+1$ vertices, and let $a$ be an arc of $\Gamma$. 
Then $0 \le i(a)\le m-1$, and:
\begin{itemize}
\item $b(a)=m-1-i(a)$;
\item $d(a)=m-i(a)$; and
\item $o(a)=i(a)$.
\end{itemize}
\end{lem}

\begin{proof}
Suppose that $a=(i,j)$. Then $j$ has $m-1$ other inneighbours, each of which forms either a ``bypass" or ``in" triangle from the perspective of $a$ (when put together in an induced subdigraph with $a$). Furthermore, none of the $m$ outneighbours of $j$ forms a ``bypass" or ``in" triangle with $a$ from the perspective of $a$. Thus, $i(a)+b(a)=m-1$, so $b(a)=m-1-i(a)$, and $0 \le i(a)\le m-1$.

Similarly, if we consider the $m$ inneighbours of $i$, each forms either a directed or ``in" triangle from the perspective of $a$, and none of the outneighbours of $i$ can form a directed or ``in" triangle with $a$ from the perspective of $a$. Thus $i(a)+d(a)=m$, and so $d(a)=m-i(a)$.

Finally, if we consider the $m$ outneighbours of $j$, each forms either a directed or ``out" triangle from the perspective of $a$, and none of the inneighbours of $j$ can form a directed or ``out" triangle with $a$ from the perspective of $a$, so $o(a)+d(a)=m$. Putting this together with the previous conclusion, we see that $o(a)=i(a)$.
\end{proof}

In the remainder of this paper, we will repeatedly need to focus on two sorts of directed cycles in the tournaments we study: directed cycles of length $3$, as considered in this section, and the directed Hamilton cycles $C_i^+$ and $C_i^-$ for various values of $i$. There may be many other directed cycles in a tournament, but we ignore all of these. For clarity, whenever we are referring to a directed cycle of length $3$ we call it a directed triangle. Whenever we refer to a directed cycle containing a particular vertex or arc, we mean the directed Hamilton cycle $C_i^+$ or $C_i^-$ for some $i$.

There are a couple of particularly important consequences of~\Cref{triangle-types} that we state as a corollary for clarity and ease of reference. These arise from noting that the formulas for $d(a)$ and $b(a)$ yield $d(a)=b(a)+1$.

\begin{cor}\label{cor-triangle-types}
Let $\Gamma$ be a regular tournament on $2m+1$ vertices, and let $a$ be an arc of $\Gamma$. 
Then $d(a) \ge 1$. Furthermore if $b(a) \ge 1$, then $d(a) \ge 2$.
\end{cor}

In the remainder of this paper, we will show that for certain signature types, in the resulting Walecki tournament we can show that for every arc $a$ such that $d(a)=1$, the third vertex of the directed triangle is $*$. This implies that any such arc must be mapped to such an arc by every automorphism of the tournament. Furthermore, the unique directed triangle that contains such an arc must also be mapped to another such directed triangle by every automorphism of the tournament, so the third vertex of such a triangle must be mapped to the third vertex of such a triangle. Since the third vertex must be $*$ in both triangles, this implies that in these Walecki tournaments, every automorphism fixes $*$.

\section{Arcs that include $*$ or consecutive vertices}

Our goal in this section is to show that if an arc includes the vertex $*$ or lies between two consecutive vertices, then it must lie in more than one directed triangle. When we say that two vertices are consecutive, we mean that they are identified with consecutive elements of $\mathbb Z_n$.

\begin{lem}\label{star}
Let $a$ be an arc in a Walecki tournament such that one of the endpoints of $a$ is $*$. Then $d(a) >1$.
\end{lem}

\begin{proof}
Towards a contradiction, suppose that $a=(*,i)$ is in exactly one directed triangle, for some $i \in \mathbb Z_n$. If $(*,i-1)$ is an arc of $W_u$ then $(i-1,i)$ is also an arc of $W_u$ (these are both in $C_{i-1}^+$) and this produces a triangle in which $a$ is a bypass arc, so $b(a) \ge 1$, and by~\Cref{cor-triangle-types} $d(a) \ge 2$, the desired contradiction. So we must have the arcs $(i,i-1)$ and $(i-1,*)$ in $W_u$. Thus $(*,i,i-1)$ is a directed triangle containing $a$. 

To avoid $(i+m,*,i)$ being a second directed triangle containing $a$, we must have an arc from $i+m$ to $i$. This arc comes from $C_{i+k}^+$, in which it is followed by the arc $(i,i+m+1)$. Again to avoid another directed triangle containing $a$, we must have the arc $(*,i+m+1)$ and therefore the cycle $C_{i+1}^-$, which also contains the arc $(i+1,*)$. But since $(i,i+1)$ lies in $C_i^+$, we now have the directed triangle $(*,i,i+1)$ containing $a$. Thus it is not possible for $a$ to be in just one directed triangle.

Reversing the directions of each arc in the above argument shows that $(i,*)$ also cannot lie in exactly one directed triangle.
\end{proof}

We will often use the following fact in our arguments about Walecki tournaments. Details of the arcs in any cycle in a Walecki tournament also appear in~\cite{ales:1,ales:2,alspach:1}, and this can be deduced from those but is not hard to work out directly. In the next result and several others, it is important to note that since $n=2m$ is even, it makes sense to consider the parity of an element of $\mathbb Z_n$.

\begin{lem}\label{alternate}
Let $i \in \mathbb Z_n$ be a vertex in a Walecki tournament, and let $j$ be any other vertex such that $j$ has the same parity as $i$. Then of the arcs between the vertex $i$ and the vertices $j$ and $j+1$, one is oriented toward $i$ and the other away from $i$.
\end{lem}

\begin{proof}
The arc between $i$ and $j$ lies in the same directed Hamilton cycle $C_\ell^+$ or $C_\ell^-$ (for some $\ell$) as the arc between $i$ and $j+1$. This means that one of the arcs must be oriented toward $i$, and the other away from $i$.
\end{proof}

This allows us to deal with the case where the endpoints of $a$ differ by $1$ (i.e., are consecutive).

\begin{lem}\label{consec}
Let $a$ be an arc in a Walecki tournament whose endpoints are $i$ and $i+1$. If $n>4$, then $d(a)>1$.
\end{lem}

\begin{proof}
Let $\ell$ be a vertex with the same parity as $i$, with $\ell \neq i$. By~\Cref{alternate}, exactly one of $i$ and $i+1$ is an outneighbour of $\ell$. If $b(a)\ge 1$ by~\Cref{cor-triangle-types} $d(a)>1$ completing the proof if $a$ is a bypass arc in the triangle induced by $\ell$, $i$, and $i+1$. The fact that there is a directed path of length $2$ via $\ell$ between $i$ and $i+1$ therefore forces this triangle to be a directed triangle. 

Since $n>4$ there are at least two vertices distinct from $i$ that have the same parity as $i$, so applying the above argument to each yields at least two directed triangles containing $a$, completing the proof.
\end{proof}

\section{Arcs whose endpoints have opposite parity}

We have already addressed the situation of consecutive vertices. In this section we consider every other situation in which an arc whose endpoints have opposite parity might lie in a unique directed triangle. We begin with a lemma that demonstrates a situation that
 often produces a second directed triangle if one exists.

\begin{lem}\label{consec-arcs}
Let $W_u$ be a Walecki tournament, and let $a$ be an arc of $W_u$ whose endpoints $i$ and $j$ have opposite parity and are not consecutive. Suppose there is some $\ell$ such that $\ell, \ell+1 \neq i,j$ and $\ell$ and $\ell+1$ are either both inneighbours or both outneighbours of $i$. Then either $d(a)>1$ or $\ell=2j-i-1$.
\end{lem}

\begin{proof}
Since $\ell$ and $\ell+1$ are either both inneighbours or both outneighbours of $i$, by~\Cref{alternate} $\ell$ must not have the same parity as $i$, so $\ell$ has the same parity as $j$ and exactly one of $\ell$ and $\ell+1$ is an outneighbour of $j$. Thus there is a directed path of length $2$ between $i$ and $j$ via either $\ell$ or $\ell+1$. Putting this together with $a$ produces either a triangle in which $a$ is a bypass arc, in which case $b(a)\ge 1$ and by~\Cref{cor-triangle-types} $d(a)>1$, or a directed triangle. Thus, if we have not yet reached our desired conclusion, then either $\ell$ or $\ell+1$ together with $a$ induce a directed triangle.

Notice that the arc between $i$ and $\ell$ is in the same Hamilton cycle as the arc between $j$ and $\ell-j+i$; likewise, the arc between $\ell+1$ and $i$ is in the same Hamilton cycle as the arc between $j$ and $\ell-j+i+1$. So unless $\ell-j+i=j$ or $\ell-j+i+1=j$, we have $\ell-j+i$ and $\ell-j+i+1$ are either both inneighbours or both outneighbours of $j$. Since $i$, $j$, and $\ell$ are distinct, we cannot have $\ell-j+i=i$. Since $\ell$ has the same parity as $j$, $\ell-j+i$ has the same parity as $i$, and in particular $\ell-j+i+1\neq i$. Furthermore, by~\Cref{alternate}, exactly one of $\ell-j+i$ and $\ell-j+i+1$ is an outneighbour of $i$. Thus we have a directed path of length two between $i$ and $j$ via either $\ell-j+i$ or $\ell-j+i+1$. We conclude that either $\ell-j+i$ or $\ell-j+i+1$ together with $a$ induce a triangle in which either $a$ is a bypass arc (our desired conclusion), or the triangle is directed. 

Since $i$ and $j$ are distinct, we cannot have $\ell-j+i=\ell$, or $\ell-j+i+1=\ell+1$. Since $i$ and $j$ are not consecutive, we cannot have $\ell-j+i=\ell+1$, or $\ell-j+i+1=\ell$. Thus, the two directed triangles we have found are distinct, and we conclude $d(a)>1$ as desired.

The only remaining possibility is that the ``unless" condition we assumed was false to find the second directed triangle, is in fact true: that is, $\ell-j+i=j$ or $\ell-j+i+1=j$. Since $\ell$ and $i$ have opposite parity, we cannot have $\ell-j+i=j$, so we must have $\ell-j+i+1=j$, and therefore $\ell=2j-i-1$, completing the proof.
\end{proof}

We now make use of the preceding lemma to deal with many possible choices for the third vertex of a unique directed triangle.

\begin{lem}\label{odd-most-cases}
Let $W_u$ be a Walecki tournament, and let $a$ be an arc of $W_u$ whose endpoints $i$ and $j$ have opposite parity and are not consecutive. Suppose $i$, $j$, and $\ell$ induce a directed triangle in $W_u$. Then one of the following holds:
\begin{itemize}
\item $d(a) >1$;
\item $\ell \in \{i-1,i+1,j-1,j+1\}$;
\item $2m+1 \equiv 0 \pmod{3}$, and we can choose $i', j'$ such that $\{i',j'\}=\{i,j\}$ and $j'=i'+(2m+1)/3$, and $\ell=2j'-i'-1$; or
\item $4m+1 \equiv 0 \pmod{3}$, and we can choose $i', j'$ such that $\{i',j'\}=\{i,j\}$ and $j'=i'+(4m+1)/3$, and $\ell=2j'-i'-1$.
\end{itemize}
\end{lem}

\begin{proof}
We assume that $d(a)=1$ and $\ell \notin \{i-1,i+1,j-1,j+1\}$, and deduce that one of the other conclusions must hold.
Note that $\ell$ has the same parity as exactly one of $i,j$. Since there is no distinction between $i$ and $j$ at this point, we may assume without loss of generality that $j$ and $\ell$ have the same parity; therefore, $i$ and $\ell-1$ have the same parity, and by hypothesis, $\ell-1 \neq i$. However, since this choice for $\ell$ may have caused us to interchange the labels of $i$ and $j$, any conclusions that are not equivalent in $i$ and $j$ need to be written in terms of some $i'$ and $j'$ with $\{i',j'\}=\{i,j\}$ (as we have done).

Since $\ell \notin \{i-1,i+1,j-1,j+1\}$, the vertices $i,j,\ell-1, \ell,$ and $\ell+1$ are all distinct. Since $d(a)=1$, neither $i$, $j$ and $\ell-1$, nor $i$, $j$ and $\ell+1$ can induce a directed triangle; also by~\Cref{cor-triangle-types}, neither can induce a triangle in which $a$ is a bypass arc. So from the perspective of $a$, each of these induced triangles must be either ``in" or ``out''.

By~\Cref{alternate}, exactly one of $\ell$ and $\ell+1$ is an outneighbour of $j$, and exactly one of $\ell-1$ and $\ell$ is an outneighbour of $i$. Since $i$, $j$, and $\ell$ induce a directed triangle, exactly one of $i$ and $j$ is an inneighbour of $\ell$. Putting all of this together with the fact that the induced triangles involving $a$ and $\ell+1$ and $a$ and $\ell-1$ are either ``in" or ``out", we deduce that $\ell$ and $\ell-1$ are either both inneighbours of $j$ or both outneighbours of $j$, and that $\ell$ and $\ell+1$ are either both inneighbours of $i$ or both outneighbours of $i$. Now we apply~\Cref{consec-arcs} to each of these.

Applying~\Cref{consec-arcs} to the arcs between $j$ and both $\ell-1$ and $\ell$, since $d(a)=1$ we conclude that $\ell-1=2i-j-1$, so $\ell=2i-j$. Applying~\Cref{consec-arcs} to the arcs between $i$ and both $\ell$ and $\ell+1$, we conclude that $\ell=2j-i-1$. Combining these yields $3j=3i+1$.
These equalities are actually equivalencies modulo $n=2m$, and clearly force $2m$ not to be $0$ modulo $3$. If $2m$ is $2$ modulo $3$, then $2m+1\equiv 0\pmod{3}$, and $3j=3i+2m+1$ which implies $j=i+(2m+1)/3$, and $\ell=2j-i-1$. Recalling that we may have to reverse the roles of $i$ and $j$, this is the first of our remaining two conclusions. If $2m$ is $1$ modulo $3$ then $4m$ is $2$ modulo $3$ so $4m+1 \equiv 0 \pmod{3}$ and similar calculations yield the final conclusion.
\end{proof}

The preceding lemma left a few cases remaining to be dealt with, one of which is the possibility that $\ell$ is one of $i-1$, $i+1$, $j-1$, or $j+1$. We address this next.

\begin{lem}\label{j-ell-consec}
Let $W_u$ be a Walecki tournament, and let $a$ be an arc of $W_u$ whose endpoints $i$ and $j$ have opposite parity and are not consecutive. Let $\ell \in \{j-1,j+1\}$ and suppose that $i$, $j$, and $\ell$ induce a directed triangle in $W_u$. Then $d(a)>1$.
\end{lem}

\begin{proof}
Suppose first that $\ell=j+1$. Since $\ell$ and $i$ have the same parity, by~\Cref{alternate} exactly one of $\ell$ and $\ell+1$ is an outneighbour of $i$ (since $i$ and $j$ are not consecutive and due to parity, neither of these vertices can be $i$). If exactly one of $\ell$ and $\ell+1$ is an outneighbour of $j$, then since $i$, $j$, and $\ell$ induce a directed triangle, there are directed paths of length $2$ in opposite directions between $i$ and $j$ via $\ell$ and via $\ell+1$. But this implies that $i$, $j$, and $\ell+1$ induce a triangle in which $a$ is a bypass arc, so by~\Cref{cor-triangle-types}, $d(a)>1$ and we are done.

We may therefore assume that $\ell$ and $\ell+1$ are either both inneighbours of $j$, or both outneighbours of $j$. Now by~\Cref{consec-arcs}, either $d(a)>1$ and we are done, or $\ell=2j-i-1$. Since $\ell=j+1$, this implies $2j-i-1=j+1$, so $j=i+2$, contradicting the distinct parities of $i$ and $j$.

Now suppose that $\ell=j-1$. Now $j$ and $j-2=\ell-1$ have the same parity, so by~\Cref{alternate} exactly one of $\ell-1$ and $\ell$ is an outneighbour of $j$. If exactly one of $\ell$ and $\ell-1$ is an outneighbour of $i$, then since $i$, $j$, and $\ell$ induce a directed triangle, there are directed paths of length $2$ in opposite directions between $i$ and $j$ via $\ell$ and via $\ell-1$. But this implies that $i$, $j$, and $\ell-1$ induce a triangle in which $a$ is a bypass arc, so by~\Cref{cor-triangle-types}, $d(a)>1$ and we are done.

We may therefore assume that $\ell$ and $\ell-1$ are either both inneighbours of $i$, or both outneighbours of $i$. Now by~\Cref{consec-arcs}, either $d(a)=1$ and we are done, or $\ell-1=2j-i-1$, meaning $\ell=2j-i$. Since $\ell=j-1$, this implies $2j-i=j-1$, so $j=i-1$, but this contradicts our hypothesis that $i$ and $j$ are not consecutive.
\end{proof}

Our next two results deal with the other cases that were not addressed in~\Cref{odd-most-cases}.

\begin{lem}\label{thirds-1}
Let $W_u$ be a Walecki tournament. Suppose that $2m+1 \equiv 0 \pmod 3$, $(2m+1)/3$ is odd, and $i$, $j$, and $\ell$ are such that $j=i+(2m+1)/3$ and $\ell=2j-i-1$. Let $a$ be the arc in $W_u$ whose endpoints are $i$ and $j$. If $i$, $j$, and $\ell$ induce a directed triangle, then $d(a)>1$.
\end{lem}

\begin{proof}
For concreteness, let us assume that there are arcs from $i$ to $j$ to $\ell$ to $i$. Note that the parity of $\ell$ is different from that of $i$, and therefore the same as that of $j$. By~\Cref{alternate}, there is an arc from $j-1$ to $i$. 

Note that since $\ell=2j-i-1$ the arc from $\ell$ to $i$ is in the same Hamilton cycle as the arc between $j=\ell-j+i+1$ and $i+j-i-1=j-1$. Accordingly, this arc must be directed from $j$ to $j-1$. Now we have a directed triangle from $j$ to $j-1$ to $i$ to $j$, so $d(a)>1$. Reversing all of the arcs gives the same conclusion.
\end{proof}

\begin{lem}\label{thirds-2}
Let $W_u$ be a Walecki tournament. Suppose that $4m+1 \equiv 0 \pmod 3$, $(4m+1)/3$ is odd, and $i$, $j$, and $\ell$ are such that $j=i+(4m+1)/3$ and $\ell=2j-i-1$. Let $a$ be the arc in $W_u$ whose endpoints are $i$ and $j$. If $i$, $j$, and $\ell$ induce a directed triangle, then $d(a)>1$.
\end{lem}

\begin{proof}
For concreteness, let us assume that there are arcs from $i$ to $j$ to $\ell$ to $i$. Note that the parity of $\ell$ is different from that of $i$, and therefore the same as that of $j$. By~\Cref{alternate}, there is an arc from $\ell+1$ to $j$. If there is an arc from $i$ to $\ell+1$ then $a$ is a bypass arc in the triangle induced by $a$ and $\ell+1$, so $b(a) \ge 1$ and by~\Cref{cor-triangle-types}, $d(a)>1$, completing the proof. So we may assume that there is an arc from $\ell+1$ to $i$.

By~\Cref{alternate}, there is also an arc from $\ell$ to $j+1$; that is, from $i+(2m+2)/3-1$ to $i+(4m+1)/3+1$. In the same Hamilton cycle and parallel to this arc, there is an arc from $(i+(2m+2)/3-1)-((2m+2)/3-2)$ to $(i+(4m+1)/3+1)+((2m+2)/3-2)$; that is, from $i+1$ to $i$. If there were also an arc from $j$ to $i+1$ then $a$ together with $i+1$ would induce a second directed triangle containing $a$, completing the proof. So we may assume that there is an arc from $i+1$ to $j$. There must then also be the parallel arc from $i$ to $j+1$ in the same Hamilton cycle.

Recall that there is an arc from $\ell+1$ to $i$, and therefore by~\Cref{alternate} there is an arc from $i$ to $\ell+2$; that is, from $j-(4m+1)/3$ to $j+(4m+1)/3+1$. In the same Hamilton cycle, there is a parallel arc from $(j-(4m+1)/3)-((2m+2)/3-2)$ to $(j+(4m+1)/3+1)+((2m+2)/3-2)$; that is, from $j+1$ to $j$. But now $a$ is a bypass arc in the induced triangle on $a$ and $j+1$, so $b(a) \ge 1$ and by~\Cref{cor-triangle-types}, $d(a)>1$. This completes the proof.
\end{proof}

To this point, we have shown that if an arc whose endpoints have opposite parity lies in a unique directed triangle, the third vertex of that triangle cannot be anything but $*$. In the final result of this section, we show that the only situation in which an arc whose endpoints have opposite parity can lie in a unique directed triangle is if $m$ is odd and our tournament is isomorphic to $W_{1_m}$. In $W_{1_m}$ the arc between $0$ and $m$ does lie in a unique directed triangle whose third vertex is $*$, but is the only arc whose endpoints have opposite parity that lies in a unique directed triangle. (You may recall that when $m$ is odd the automorphism group of $W_{1_m}$ is the cyclic group generated by $\sigma$; while this does map the arc between $0$ and $m$ to other arcs, the endpoints of any of these arcs have the same parity.)

\begin{lem}\label{parallel}
Let $W_u$ be a Walecki tournament, and let $a$ be an arc of $W_u$ whose endpoints $i$ and $j$ have opposite parity. If $i$, $j$ and $*$ induce a directed triangle in $W_u$, then either $d(a)>1$, or $j=i+m$ and $u\in \{1_mR_i,1_mR_j\}$.
\end{lem}

\begin{proof}
By considering instead the isomorphic tournament $W_{uR_{-i}}$ or $W_{uR_{-j}}$ if necessary, we may assume $i=0$ and $1 \le j \le m$ is odd. 

By~\Cref{cor-triangle-types}, if $b(a) \ge 1$ then we are done; also if any of the triangles involving $i$, $j$, and $\ell$ for $\ell \in \mathbb Z_n$ with $\ell \neq i,j$ is directed then we are done. So every vertex in $\{j+1, \ldots, 2m-1\}$ must be in either an ``in" triangle or an ``out" triangle with $a$, from the perspective of $a$. This means that each of these vertices is either a mutual inneighbour of $0$ and $j$, or a mutual outneighbour of $0$ and $j$.

For $r \in \{j+1,\ldots, 2m-2\}$,~\Cref{alternate} tells us that $r$ is a mutual inneighbour of $0$ and $j$ if and only if $r+1$ is a mutual outneighbour of $0$ and $j$ (we apply the lemma to either $0$ or $j$ depending on the parity of $r$). This implies that all the vertices of one parity in $\{j+1, \ldots, 2m-1\}$ are mutual inneighbours of $0$ and $j$, while all the vertices of the other parity are mutual outneighbours of $0$ and $j$.

If $j \le m-1$ then the arc between $2m-1$ and $0$ is parallel to and in the same Hamilton cycle as the arc between $2m-1-j$ and $j$. Thus $2m-1$ and $2m-1-j$ are either both mutual inneighbours, or both mutual outneighbours of $0$ and $j$. But $2m-1$ is odd, and $2m-1-j$ is even; this is a contradiction that completes the proof in this situation.

The possibility remains that $j=m$. In this case, the arc between $2m-1$ and $0$ is parallel to and in the same Hamilton cycle as the arc between $j=m$ and $j-1=m-1$. In this case, however, this means that exactly one of $2m-1$ and $m-1$ is a mutual inneighbour of $0$ and $j$. Note that these vertices have opposite parity. Moreover, the same argument as above now applied to the set $\{1, \ldots, m-1\}$ of vertices, tells us that all of the vertices of one parity in this set are mutual inneighbours of $0$ and $j$, while all the vertices of the other parity are mutual outneighbours of $0$ and $j$.

Putting these together, we see that either all of the even vertices in $\mathbb Z_n$ except $0$ are mutual outneighbours of $0$ and $j$, while $1$ is a mutual inneighbour of $0$ and $j$; or they are all mutual inneighbours of $0$ and $j$, while $1$ is a mutual outneighbour of $0$ and $j$. The former case implies that there is an arc from $0$ to $2\ell$ for every $1\le \ell \le m-1$, while the latter implies the arcs are in the opposite direction. In the first case, $u_\ell=0$ for $0 \le \ell \le m-1$; in the second case, $u_\ell=1$ for $0 \le \ell \le m-1$. So we either have $u=0_m$ or $u=1_m$. After applying $R_i$ or $R_j$ to $u$ to return to the original tournament, we conclude $u\in \{1_mR_i,1_mR_j\}$, as desired.
\end{proof}

\section{Arcs whose endpoints have the same parity}

In this section, we consider the remaining possible type of arc: arcs whose endpoints have the same parity. We have already seen that when $m$ is odd, $W_{1_m}$ has a number of arcs that lie in unique directed triangles. When the endpoints of an arc have opposite parity, the information~\Cref{alternate} provides about the outneighbours and inneighbours of one endpoint complements the information provided by the other endpoint. When the endpoints have the same parity, both provide the same information. This makes it  much harder to pin down which arcs whose endpoints have the same parity can be in unique directed triangles. In particular, there seem to be many possible Walecki tournaments that have some arc $a$ whose endpoints have opposite parity, and $d(a)=1$. The amazing thing that is not so difficult to prove, though, is that in all cases the third vertex of the unique directed triangle must be $*$. For our purposes, this is all we need.

We begin with a preliminary result that narrows down the possible third vertices.

\begin{lem}\label{mults-of-i}
Let $a$ be an arc in a Walecki tournament $W_u$ whose endpoints are $i,j \in \mathbb Z_n$, where $j-i \le m$ and $j$ and $i$ have the same parity. Let $t$ be the additive order of $j-i$ in $\mathbb Z_n$. If $d(a)=1$ then the other vertex of the directed triangle containing $a$ lies in $\{j+(j-i), j+2(j-i),\ldots, j+(t-2)(j-i),*\}$.
\end{lem}

\begin{proof}
To simplify our notation and arguments, we will work in $W_{uR_{-i}}$ so that we can take $i=0$ and $j-i=j$ as the endpoints of $a$, and the set of possible third vertices becomes $\{2j, \ldots, (t-1)j,*\}$.

Our goal is to show that it is not possible for all of the triangles containing $a$ whose other vertex lies in $\{2j, \ldots, (t-1)j,*\}$ to be either ``in" or ``out" from the perspective of $a$. This implies that either one of them is a ``bypass" triangle, in which case $b(a) \ge 1$ and by~\Cref{cor-triangle-types} $d(a)\ge 2$, a contradiction, or one of them is directed, and must therefore be the unique directed triangle containing $a$. Therefore, towards a contradiction, suppose that all of the triangles containing $a$ whose other vertex lies in $\{2j, \ldots, (t-1)j,*\}$ are either ``in" or ``out" from the perspective of $a$. 

If $j=m$ then the arcs between $*$ and each of $0$ and $j$ are both in $C_0^+$ or in $C_0^-$, and one must begin at $*$ while the other ends at $*$, producing an immediate contradiction. Henceforth we assume $j<m$.

In the argument that follows, we may reverse the direction of all arcs and reach the same conclusion. So we begin by assuming without loss of generality that there are arcs from $*$ to both $0$ and $j$; that is, $u_0=1=u_j$. This implies that there is an arc from $j$ to $-j$, and that there is an arc from $2j$ to $0$.

Our assumption that each triangle is ``in" or ``out" from the perspective of $a$, allows us to conclude that for every $2 \le s \le t-1$, $sj$ is either a mutual inneighbour or a mutual outneighbour of $0$ and $j$. This in turn is equivalent to the existence of an arc parallel to that between $j$ and $sj$ (from the same Hamilton cycle) between $0$ and $(s+1)j$. Based on our initial choice of directions, it turns out at each step (inductively) that if after reducing modulo $n$ we have $0<(s+1)j<j$, then this arc goes from $0$ to $(s+1)j$; otherwise it goes from $(s+1)j$ to $0$. Since $0<j<m$, we must have $j<(t-1)j<n$ after reducing modulo $n$. Thus we eventually conclude that there is an arc from $-j=(t-1)j$ to $0$. But this contradicts the existence of the arc from $j$ to $-j$, completing the proof.
\end{proof}

We can now show that the third vertex must in fact be $*$.

\begin{lem}\label{even}
Let $a$ be an arc in a Walecki tournament whose endpoints are $i,j \in \mathbb Z_n$, where $j-i \le m$ and $j$ and $i$ have the same parity. If $d(a)=1$, then the other vertex of the directed triangle containing $a$ is $*$.
\end{lem}

\begin{proof}
Again, to simplify our notation and arguments, we will work in $W_{uR_{-i}}$ so that we can take $i=0$ and $j-i=j$ as the endpoints of $a$.
By~\Cref{mults-of-i}, the other vertex of the directed triangle lies in $\{2j,\ldots, (t-1)j,*\}$, where $t$ is the additive order of $j$ in $\mathbb Z_n$. 

Towards a contradiction, suppose that $(0,j,sj)$ is a directed triangle for some $2 \le s \le t-1$. (Reversing the direction of this cycle and of all subsequent arcs in the argument yields the same conclusion.) Since $j$ is even, $sj$ is also even. 

We distinguish two possibilities, depending on whether after reduction modulo $n$ we have $0<sj<j$, or $j<sj<n$. 

Suppose first that $j<sj<n$.
Since there is an arc from $sj$ to $0$, we must have $u_{sj/2}=1$, and there is also an arc from $0$ to $sj+1$. Since there is an arc from $j$ to $sj$, we must have $u_{(s+1)j/2}=0$, and there is also an arc from $sj+1$ to $i$. But now $a$ is a bypass arc in the triangle on $0$, $j$, and $sj+1$, meaning $b(a) \ge 1$ so $d(a) \ge 2$ by~\Cref{cor-triangle-types}, a contradiction.

Now suppose $0<sj<j$. Since there is an arc from $sj$ to $0$,  we must have $u_{sj/2}=0$, and there is also an arc from $0$ to $sj+1$. Since there is an arc from $j$ to $sj$, we must have $u_{(s+1)j/2}=1$, and there is also an arc from $sj+1$ to $j$. But now $a$ is a bypass arc in the triangle on $0$, $j$, and $sj+1$, meaning $b(a) \ge 1$ so $d(a) \ge 2$ by~\Cref{cor-triangle-types}, again a contradiction.

Since there is no $2 \le s\le t-1$ such that $(0,j,sj)$ can be a directed triangle (in either direction),~\Cref{mults-of-i} implies that the other vertex of the directed triangle containing $a$ must be $*$.
\end{proof}

\section{Automorphisms and isomorphisms of Walecki tournaments}

We begin this section by producing a result that summarises the results of the previous sections. Note that the following result is not true for the (unique up to isomorphism) Walecki tournament on $5$ vertices, which does have at least one arc that lies in a unique directed triangle whose third vertex is not $*$.

\begin{thm}\label{unique-star}
Suppose that $a$ is an arc in a Walecki tournament $W_u$ on at least $7$ vertices that is in exactly one directed triangle. Then the third vertex of that directed triangle is $*$.

Furthermore, either $m$ is odd, $u=1_mR_i$ for some $i$, and the endpoints of $a$ are $i$ and $i+m$, or the endpoints of $a$ are elements of $\mathbb Z_n$ that have the same parity.
\end{thm}

\begin{proof}
If either endpoint of $a$ is $*$, this is~\Cref{star}. If the endpoints of $a$ are consecutive then since $n \ge 6$, this is~\Cref{consec}. If the endpoints of $a$ have opposite parity but are not consecutive, then this follows from one of~\Cref{odd-most-cases},~\Cref{j-ell-consec},~\Cref{thirds-1}, or~\Cref{thirds-2}, together with~\Cref{parallel} to complete the ``furthermore". Finally, if the endpoints of $a$ have the same parity then this follows from~\Cref{even}.
\end{proof}

\begin{cor}\label{main-cor}
Suppose that the Walecki tournament $W_u$ on at least $7$ vertices contains an arc $a$ that lies in a unique directed triangle. Then every automorphism of $W_u$ fixes $*$. Moreover, if $W_u \cong W_v$, then any isomorphism must map the vertex labelled $*$ in $W_u$ to the vertex labelled $*$ in $W_v$.
\end{cor}

\begin{proof}
Since an automorphism is an isomorphism from $W_u$ to itself, the second statement implies the first. Suppose, then, that $W_u$ contains an arc $a$ that lies in a unique directed triangle. By~\Cref{unique-star}, this triangle must have $*$ as its third vertex. Any isomorphism from $W_u$ to $W_v$ must map $a$ to some arc $a'$ in $W_v$ that lies in a unique directed triangle. Furthermore, it must map the unique directed triangle containing $a$ to the unique directed triangle containing $a'$. By~\Cref{unique-star}, the third vertex of the unique directed triangle containing $a'$ must be the vertex of $W_v$ that is labelled $*$. Thus our isomorphism must map the vertex labelled $*$ in $W_u$ to the vertex labelled $*$ in $W_v$.
\end{proof}

In the next and final section of this paper, we define a fairly significant family of Walecki tournaments that do contain an arc that lies in a unique directed triangle. 

\section{Walecki Tournaments in which $*$ is uniquely determined}

We begin by defining a family of signatures. 

\begin{defn}
Let $m$ be even, say $m=2k$. Let $\mathcal S$ be the set of binary strings $u$ of length $m$ that have the following properties:
\begin{itemize}
\item $u=u_0 \ldots u_{m-1}$; and
\item for $0 \le i \le k-1$, $u_{i+k} \neq u_{i}$. 
\end{itemize}
So we can pick any binary string of length $k$ for the first $k$ entries, but the remaining entries are completely determined by those first $k$ entries. 
\end{defn}

Now we show that when $u \in \mathcal S$, the arc between $0$ and $m$ in $W_u$ is in exactly one directed triangle.

\begin{thm}
Let $m=2k$ with $k\ge 3$, and $u \in \mathcal S$. In $W_u$, if $a$ is the arc between $0$ and $m$, then using the notation of~\Cref{triangle-types}, $d(a)=1$.

In particular, this means that every automorphism of $W_u$ fixes $*$. Moreover, if $W_u \cong W_v$, then any isomorphism must map the vertex labelled $*$ in $W_u$ to the vertex labelled $*$ in $W_v$.
\end{thm}

\begin{proof}
Note that each arc in $W_u$ arises from a directed version of either the cycle $C_j$ or the cycle $C_{j+k}$ for some $0 \le j \le k-1$. So let $0 \le j \le k-1$. If $u_j=1$ then every arc in $C_j^+$ that does not involve $*$ has the form $(j-\ell,j+1+\ell)$ for some $0 \le \ell \le m-1$, or $(j+1+\ell,j-1-\ell)$ for some $0 \le \ell \le m-1$; if $u_j=0$ then $C_j^-$ has the same underlying edges but each arc has the opposite direction. Likewise, if $u_{j+k}=0$ then every arc in $C_{j+k}^-$ that does not involve $*$ has the form $(j+k+1+\ell,j+k-\ell)$ for some $0 \le \ell \le m-1$, or $(j+k-1-\ell,j+k+1+\ell)$ for some $0 \le \ell \le m-1$; if $u_{j+k}=1$ then $C_{j+k}^+$ has the same underlying edges but each arc has the opposite direction. 

Recall that since $u \in \mathcal S$, we have $u_j \neq u_{j+k}$. The formulas of the previous paragraph tell us that if $u_j=1$ then the arcs involving the vertex $0$ in $C_j^+$ are $(2j,0)$ (if $j=0$ then this is replaced by $(*,0)$) and $(0,2j+1)$, and in $C_{j+k}^-$ are $(2j+m+1,0)$ and $(0,2j+m)$ (and the reverse of these arcs if $u_j=0$). Meanwhile, the arcs involving the vertex $m$ in $C_j^+$ are $(2j+m+1,m)$ and $(m,2j+m)$, and in $C_{j+k}^-$ are $(2j,m)$ and $(m,2j+1)$ (and the reverse of these arcs if $u_j=0$). Since every vertex other than $*$ has one of the forms $2j+1$, $2j$, $m+2j+1$ or $m+2j$ for some $0 \le j \le k-1$, we see that because $u \in \mathcal S$, for each vertex $i$ of our tournament other than $*$, $0$, and $m$, we either have arcs from both $0$ and $m$ to $i$, or arcs from $i$ to both $0$ and $m$. Thus except for the triangle involving $*$, every triangle that includes $0$ and $m$ is either an ``in" triangle or an ``out" triangle.

Thus, $o(a)=i(a)=m-1$, so using~\Cref{triangle-types}, $d(a)=1$ (the triangle involving $*$ is the directed triangle). The final conclusion is an immediate consequence of~\Cref{main-cor}
\end{proof}

It is worth pointing out that $1_{2k}R_k \in \mathcal S$, so at least some of the Walecki tournaments we have identified in these results are isomorphic to those whose automorphism groups were already known to be trivial through the work of Ales. However, his result did not directly show that there could not be a Walecki tournament $W_v$ such that $W_{1_m} \cong W_v$ but the vertex $*$ of $W_{1_m}$ maps to some vertex not labeled $*$ in $W_v$. Furthermore, there are definitely Walecki tournaments whose signature lies in $S$ that are not isomorphic to $W_{1_m}$ for the appropriate $m$. Specifically, the signature of $W_{1001}$ lies in $\mathcal S$, but it can easily be checked computationally that the automorphism group of $W_{1001}$ is cyclic of order $3$, so $W_{1001}$ cannot be isomorphic to $W_{1111}$, whose automorphism group is trivial.

It may be possible to further extend these ideas. This may be possible very directly by finding other families of signatures whose Walecki tournaments include an arc that lies in a unique directed triangle.
It may take a more indirect approach, for example by looking at the other end of the possible values for the parameters we have studied here, and finding families of signatures whose Walecki tournaments include a unique arc $a$ with the property that $i(a)=o(a)=0$. It may require a more complex approach such as counting the numbers of arcs that lie in a particular number of directed triangles. Additional research along any of these lines would be of interest. 

It is not the case that every Walecki tournament has an arc that is in exactly one directed triangle, whether $m$ is even or odd. Neither $W_{10001}$ nor $W_{110011}$ has such an arc. So~\Cref{main-cor} does not apply to all Walecki tournaments.

The family $\{W_u: u \in \mathcal S\}$ does not include any Walecki tournaments whose signature has odd length. However, when the signature has even length $2k$, it covers $2^k$ of the possible $2^{2k}$ signatures (so long as $k \ge 3$). If combined with the isomorphisms produced by the complementing register shift $R_1$, it covers even more. For example, when $k=3$ there are $2^3=8$ of the $2^6=64$ signatures in $\mathcal S$, but applying various powers of $R_1$ results in a total of $48$ signatures ($4$ of the $6$ isomorphism classes under the action of $R_1$).

\end{document}